\documentclass[12pt,reqno]{amsart}
\usepackage[T1]{fontenc}
\usepackage[english]{babel}
\usepackage[utf8]{inputenc}
\usepackage{amssymb}
\usepackage{amsmath}
\usepackage{amsfonts}
\usepackage{amsthm}
\usepackage{csquotes}
\usepackage{comment}
\usepackage[square,numbers]{natbib}

\usepackage[a4paper,margin=2.5cm]{geometry}
\usepackage{graphicx} 

\usepackage{hyperref}
\hypersetup{
    colorlinks=true,
    linkcolor=blue,
    filecolor=magenta,      
    urlcolor=cyan,
}

\newcommand{\esp}{\operatorname{\mathbb{E}}}
\newcommand{\prob}{\mathbb{P}}
\newcommand{\norm}[1]{\|#1\|}
\newcommand{\R}{\mathbb{R}}
\newcommand{\tr}{\operatorname{tr}}
\newcommand{\var}{\operatorname{var}}

\def\row{{\rm row}}

\def\dksuggestion{\textcolor{red}}

\usepackage{macros}

\title[Spectrum of inhomogeneous matrices]{On spectral outliers
of inhomogeneous symmetric random matrices}
\author[D.\ J.\ Altschuler \and P.\ O.\ Santos \and K.\ Tikhomirov \and P.\ Youssef]{%
        Dylan J. Altschuler \and 
        Patrick Oliveira Santos \and 
        Konstantin Tikhomirov \and 
        Pierre Youssef
        }

\address{Dylan J. Altschuler, Department of Mathematical Sciences, Carnegie Mellon University.}
\email{daltschu@andrew.cmu.edu}
\address{Patrick Oliveira Santos. Laboratoire d'Analyse et de Math\'ematiques Appliqu\'ees UMR 8050, Universit\'e Gustave Eiffel, Universit\'e Paris Est Cr\'eteil, Marne-la-Vall\'ee, France.}
\email{patrick.oliveirasantos@u-pem.fr}
\address{Konstantin Tikhomirov, Department of Mathematical Sciences, Carnegie Mellon University.}
\email{ktikhomi@andrew.cmu.edu}
\address{Pierre Youssef. Division of Science, NYU Abu Dhabi, Abu Dhabi, UAE \& Courant Institute of Mathematical Sciences, New York University, New York, USA.}
\email{yp27@nyu.edu}

\date{}
\begin{document}
\newtheorem{theorem}{Theorem}[section]
\theoremstyle{plain}

\newtheorem{question}{Question}
\newtheorem{corollary}[theorem]{Corollary}
\newtheorem{lemma}[theorem]{Lemma}
\newtheorem{conjecture}[theorem]{Conjecture}
\newtheorem{proposition}[theorem]{Proposition}

\theoremstyle{definition}
\newtheorem{example}[theorem]{Example}
\newtheorem{definition}[theorem]{Definition}
\newtheorem{remark}[theorem]{Remark}

\newtheorem*{model}{Model}

\maketitle

\begin{abstract}
    Sharp conditions for the presence of spectral outliers are well understood for Wigner random
    matrices with iid entries.
    In the setting of \textit{inhomogeneous} symmetric random matrices (i.e.,\ matrices with a non-trivial variance profile), the corresponding problem has been considered only recently. Of special interest is the setting of sparse inhomogeneous matrices since sparsity is both a key feature and a technical obstacle in various aspects of random matrix theory. For such matrices, the largest of the variances of the entries has been used in the literature as a natural proxy for sparsity. We contribute sharp conditions in terms of this parameter for an inhomogeneous symmetric matrix with sub-Gaussian entries to have outliers. Our result implies a ``structural'' universality principle: the presence of outliers is only determined by the level of sparsity, rather than the detailed structure of the variance profile.

\end{abstract}

\section{Introduction}

Given an $n\times n$ random symmetric matrix $M_n$, 
define its {\it Empirical Spectral Distribution} (ESD) as the
random probability measure
\begin{align*}
    \mu_{M_n}:=\frac{1}{n}\sum_{k\in[n]} \delta_{\lambda_k(M_n)},
\end{align*}
where $\lambda_k(M_n)$ is the $k$-th largest eigenvalue of $M_n$. 
A central problem in random matrix theory is to
establish necessary and sufficient conditions for the
sequence $\mu_{M_n}$, $n\geq 1$,
to converge to a non-random measure (see, for example,
\cite[Section 2.4]{tao2012randommatrix} for a discussion
of different types of convergence: almost surely, in probability,
in expectation).

A celebrated result of Wigner \cite{wigner1958distribution}
asserts that the ESD of a normalized symmetric matrix
with i.i.d entries
converges weakly almost surely and in expectation to the {\it semi-circle distribution} $\mu_{sc}$ whose density $f_{sc}$ is given by  
\begin{align}\label{semicircular density}
    f_{sc}(x)=\frac{1}{2\pi}\sqrt{4-x^2}\, \mathbf{1}_{\{|x| \le 2\}}.
\end{align}
The main feature of this result is its universality: the limiting ESD does not depend on the
distribution of the matrix entries.

The universality phenomenon for the matrix spectrum
has been actively studied
for other models of randomness involving non-identically distributed entries.
Regarding convergence to the semi-circle law, we refer, in particular, to  
\cite{bogachev1992density,molchanov1992limiting,nica2002operator}
for random band matrices, \cite{erdHos2012bulk,erdos2013local}
for generalized Wigner matrices,
and to the survey \cite{benaych2016lectures} for further references.
{\it{}Sparse}
inhomogeneous random matrices have been a subject of much recent interest. 
The presence of sparsity often plays the role of both a key feature and a key challenge in statistical inference, graph theory, and random matrices \cite{bandeira2023matrix, bandeira2016sharp,brailovskaya2022universality}.

We now introduce the matrix model to be studied in this paper.

\begin{model}
Let $W_n$ denote an $n\times n$ Wigner matrix whose entries on and above the diagonal are iid copies of a centered random variable $\xi$ having unit variance. Further, 
let $\Sigma_n=(\sigma_{ij})_{1\leq i,j\leq n}$ be a symmetric $n\times n$ matrix
with non-negative entries satisfying 
$$
\sum_{j=1}^n\sigma_{ij}^2=1, \text{ for all }1\leq i\leq n\,
$$
(that is, the matrix $(\sigma_{ij}^2)_{1\leq i,j\leq n}$ is {\it{}doubly stochastic}). The  object of our study is the deformed matrix
$X_n=\Sigma_n\circ W_n$, where ``$\circ$'' denotes the Hadamard (entry-wise)
product. Note that the classical setting in Wigner's semi-circle theorem corresponds to $\Sigma_n=\frac{1}{\sqrt{n}} \mathbf{1}\mathbf{1}^t$, where $\mathbf{1}$ denotes the vector of all ones. 
\end{model}

A natural question regarding the above model is:
{\it Does the limiting ESD of $X_n$ depend on either $\Sigma_n$ or the distribution of $\xi$?}

Denote by $\sigma_n^*= \max_{1\leq i,j\leq n} \sigma_{ij}$ the maximum of the standard deviations of the entries of $X_n$. The parameter $\sigma_n^*$
can be viewed as a proxy of the matrix sparsity.
The work \cite{GNT2014} provides an essentially complete answer to the above question
(we also refer to \cite{bandeira2023matrix} for closely
related statements):

\begin{theorem}[\cite{GNT2014}]\label{th: main ESD}
    Let $W_n$ be an $n\times n$ symmetric matrix whose entries on and above the diagonal are iid copies of a centered random variable $\xi$ with unit variance. Let $\Sigma_n=(\sigma_{ij})_{1\leq i,j\leq n}$ be a symmetric matrix such that $(\sigma_{ij}^2)_{1\leq i,j\leq n}$ is doubly stochastic.
    Setting $X_n=\Sigma_n\circ W_n$ and assuming that $\sigma_n^*\to 0$, we have that $\mu_{X_n}$ converges weakly almost surely and in expectation to $\mu_{sc}$.
\end{theorem}
The main result of \cite{GNT2014} can additionally handle $W_n$ having entries which satisfy a Lindeberg-type condition rather than being identically distributed. The authors of \cite{GNT2014} use the Stieltjes transform method to establish convergence in expectation. Additionally, it was noticed in \cite{chin2023necessary} that this can be ``upgraded'' to almost sure convergence using a concentration inequality for the spectral measure of random matrices with independent entries \cite[Lemma C.2]{bordenave2011spectrum}, \cite{guntuboyina2009concentration}. 
In this note, we provide an alternative proof of Theorem~\ref{th: main ESD} based on the moment method (see Appendix to this paper).

Take $\Sigma_n$ to be the adjacency matrix of a $d$-regular graph 
(rescaled by $1/\sqrt{d}$ to have a doubly stochastic variance profile). Then $X_n$ can be viewed as a weighted adjacency matrix
with iid centered edge weights of variance $1/d$.
Specialized to this setting, the above theorem asserts that the condition $d \to \infty$ is sufficient for almost sure convergence of the limiting ESD to the semi-circle law, regardless of the structure of the underlying graph.
On the other hand, if $d\geq 2$ is fixed, then the distribution of $\mu_{X_n}$ is dependent both on the structure of $\Sigma_n$ as well as the atomic distribution of $W_n$.
For completeness, we provide a proof of this claim in the Appendix (see Proposition~\ref{prop: ESD} there).

\subsection{Main results: spectral outliers of inhomogeneous matrices}

Convergence of the ESD of $X_n$ to the semi-circle law, which is supported on $[-2,2]$,  guarantees that $n - o(n)$ eigenvalues of $X_n$ lie in the interval
$[-2,2]$. That does not rule
out existence of \textit{spectral outliers}, that is, eigenvalues
near the spectral edges which are at a non-vanishing distance
to the support of the limiting ESD. In the classical setting of Wigner matrices with iid entries, it is known
that the assumption of bounded fourth moment 
of the entries is necessary
and sufficient to guarantee the absence of outliers \cite{bai1988necessary}. 

Spectral outliers of Hermitian random matrices have been extensively studied \cite{KS2003, baik2005phase, BS2006, BGBK2019, tikhomirov2021outliers, ADK2021, brailovskaya2022universality},
in part due to important applications in statistical inference and signal processing
(we also refer, among others, to papers
\cite{TW1996, Soshnikov1999, johnstone2001largest, Sodin2009, sodin2010spectral, EYY2012, LS2018} dealing
with the limiting distribution of the extreme eigenvalues of Wigner
and sample covariance matrices). In particular, the seminal result of Baik--Ben Arous--Peche \cite{baik2005phase}, and its generalizations, unraveled a phase transition phenomenon in the appearance of outliers in the spectrum of classical random matrix ensembles under small rank perturbations. 
In \cite{tikhomirov2021outliers}, a similar phase transition phenomenon was observed in the case of Wigner matrices under random sparsification, i.e., in the setting where each entry in the Wigner matrix is kept independently with some probability $p_n$ (see also \cite{ADK2021}). In the setting of the present paper, that corresponds to studying the norm of $B_n\circ W_n$, where $W_n$ is a Wigner matrix and $B_n$ is a symmetric matrix whose entries on and above the main diagonal are iid copies of a Bernoulli random variable with parameter $p_n$. \\

In this note, we give necessary and sufficient conditions for the presence of outliers for the matrix $X_n=\Sigma_n\circ W_n$ in terms of $\Sigma_n$ and $\xi$. The main problem we want to address is
whether structural
universality phenomenon is observed for the spectral outliers, i.e, the presence/absence of outliers is characterized
only by the level of sparsity of $\Sigma_n$
and does not depend on the fine structure of the variance profile.

Unlike for the bulk, it is readily checked that the presence or absence of outliers depends on the distribution of $\xi$ regardless of the sparsity pattern. Indeed, this follows directly from the results of \cite{latala2018dimension}, where the norms of inhomogeneous matrices with Gaussian entries and with heavy-tailed entries were completely characterized. While distributional universality provably fails to hold, it remains a compelling question whether structural universality holds under some restrictions on $\xi$. We isolate the effect of the structure of $\Sigma_n$ by restricting $\xi$ to the class of sub-Gaussian distributions. Recall that a random variable $\xi$ is said to be sub-Gaussian if there exists a constant $C$ such that $\mathbb{E} e^{\lambda (\xi-\esp \xi)}\leq e^{C\lambda^2}$ for every $\lambda$. Our question is:

\begin{question}\label{q: outliers}
Let $W_n$ be an $n\times n$ Wigner matrix whose entries are
sub-Gaussian and consider $X_n=\Sigma_n\circ W_n$. Does the presence/absence of outliers for $X_n$ depend on either the specific structure of $\Sigma_n$ or the particular (sub-Gaussian) distribution of $\xi$? What conditions on $\Sigma_n$ guarantee the presence/absence of outliers?
\end{question}

We provide a sharp characterization for the appearance of outliers in terms of the sparsity proxy $\sigma_n^*$. 
Recall by Theorem~\ref{th: main ESD} that whenever $\sigma_n^*\to 0$, the limiting ESD of $X_n$ is the semi-circle distribution. Thus, if $\sigma_n^* \to 0$, the above question asks whether $\Vert X_n\Vert \to 2$, where $\Vert \cdot\Vert$ denotes the spectral norm.  

\begin{theorem}[Main result]\label{th: outliers} 
    Let $W_n$ be an $n\times n$ symmetric matrix whose entries on and above the main diagonal are iid copies of a centered sub-Gaussian random variable $\xi$ with unit variance. For any symmetric matrix $\Sigma_n=(\sigma_{ij})_{1\leq i,j\leq n}$ such that $(\sigma_{ij}^2)_{1\leq i,j\leq n}$ is doubly stochastic, if $\sigma_n^*\sqrt{\log n}\to 0$, then $\Vert \Sigma_n\circ W_n\Vert \to 2$ almost surely, implying the absence of outliers. 
\end{theorem}

Specialized to the setting of a weighted $d$-regular graph, the above theorem asserts that in the case of sub-Gaussian weights, if $d$ grows faster than $\log n$, then there are no outliers regardless of the structure of the base graph. This result is sharp as the next theorem shows.

\begin{theorem}[A sufficient condition for existence of outliers]\label{th: presence outliers}  
Let $\xi$ be a centered random variable of unit variance and with bounded
fourth moment,
and for each $n$, let $W_n$ be an $n\times n$ symmetric matrix whose entries above the main diagonal are iid copies of $\xi$.
Fix any sequence $d:= d(n)=O(\log n)$ with $d \to \infty$ such that $nd$ is even for every $n$. Then
there exists a sequence of non-random $d$-regular graphs with adjacency matrices $A_n$ such that
for all large $n$, the matrices $\frac{1}{\sqrt{d}}A_{n}\circ W_n$ have outliers almost surely.
That is,
$$\prob\bigg(\liminf_{n\to\infty}\Big\Vert \frac{1}{\sqrt{d}}A_{G_n}\circ W_n\Big\Vert>2\bigg)=1.$$
\end{theorem}

The sharp characterization provided by Theorem \ref{th: outliers} combined with Theorem \ref{th: presence outliers} was previously established in the special case of $d$-regular graphs with Gaussian edge weights \cite{bandeira2016sharp}. Moreover, a succession of works \cite{bandeira2016sharp, benaych2014largest, khorunzhiy2008estimates, latala2018dimension,  sodin2010spectral} showed that for bounded weights, there are no outliers whenever $d/\log n\to \infty$. In the other direction, the existence of outliers for Rademacher weights was previously  known when $d=O(\sqrt{\log n})$ \cite{seginer2000expected}, rather than up to the sharp rate of  $d=O(\log n)$ that our results capture. It should be noted that the argument of \cite{seginer2000expected} can be extended and adapted to the setting of Theorem~\ref{th: presence outliers} showing the existence of outliers in the regime $d/\log n \to 0$. However, that argument fails in the regime where $d$ is of order $\log n$, while the proof we provide for Theorem~\ref{th: presence outliers} gives a unified treatment which works in all regimes. We refer to Remark~\ref{rk: seginer} for more details.  
\\

Returning to the general setting of Theorem \ref{th: outliers}, to the best of our knowledge, the \textit{absence} of outliers for the general sub-Gaussian distributions has only been treated in \cite{benaych2014largest} for matrices with restricted bandwidth and graphs with suboptimal sparsity $d/\log^{9/2}n \to \infty$. However, there has been much recent progress in characterizing the spectral norm of inhomogeneous random matrices \cite{bandeira2023matrix,bandeira2016sharp, benaych2014largest, brailovskaya2022universality, latala2018dimension}. To the best of our knowledge, there are two existing upper bounds for the norm of inhomogeneous sub-Gaussian random matrices. The first is derived through a symmetrization procedure and comparison with the Gaussian setting \cite{bandeira2016sharp}, which leads to a suboptimal constant multiplicative factor and thus cannot capture sharp conditions for the presence/absence of outliers. The second result captures the correct constant $2$ but at the expense of increasing the additive error term, which also results in a suboptimal regime for the absence of outliers (see \cite[Corollary~3.6]{bandeira2016sharp}, \cite[Remark~4.13]{latala2018dimension}). Theorem~\ref{th: outliers} addresses these shortcomings and captures the correct asymptotic behavior for general sub-Gaussian matrices.

The proof of Theorem~\ref{th: outliers} avoids the use of symmetrization and instead directly estimates the trace of powers of $ \Sigma_n\circ W_n$. It uses a truncation argument to split this matrix into the sum of a matrix with bounded weights and a remainder. The former is controlled using the results in \cite{bandeira2016sharp,benaych2014largest,latala2018dimension}. 

The remainder is controlled by a new version of a compression argument, originally devised in \cite[Proposition 2.1]{bandeira2016sharp} and also used in \cite{latala2018dimension}, for comparing sparse matrices to dense matrices of smaller dimension. Implementing this argument in our setting requires tuning the dimension of the dense matrix as a function of the truncation level. 

Matching upper and lower bounds on the norms of inhomogeneous Gaussian matrices were previously given in \cite{latala2018dimension}. However, these results are only sharp up to a multiplicative constant, thus failing to capture the scale necessary for characterizing outliers. 
Upper bounds with the correct constant scale for the norm of inhomogeneous Gaussian and bounded random matrices were derived in \cite{bandeira2016sharp,latala2018dimension}. In the special case of Gaussian or bounded matrices, these results imply half of the dichotomy we prove, namely the absence of outliers. However, matching lower bounds would be needed to establish the presence of outliers. 
Theorem~\ref{th: presence outliers} addresses this shortcoming by showing that for \textit{any} sparsity $d = \cO{\log n}$,  there is a sequence of graph adjacency matrices $\Sigma_n$ such that $X_n:= \Sigma_n \circ W_n$ has outliers for $W_n$ having \textit{any} non-atomic distribution with finite fourth moment.  Thus, there is a sharp transition at sparsity $\cO{\log n}$ below which structural universality is not observed.
Below this level of sparsity, the presence/absence of outliers necessarily depends on the structure of $\Sigma_n$.
The sequence of deterministic graphs in Theorem~\ref{th: presence outliers} is a union of cliques, and the proof of the result combines an anti-concentration argument with incompressibility properties of eigenvectors of Wigner matrices \cite{rudelson2016no}.  \\

The paper is organized as follows: in Sections~\ref{sec: no outliers} and Section~\ref{sec: outliers}, we prove Theorems ~\ref{th: outliers} and \ref{th: presence outliers} respectively. Appendix~\ref{sec: universal ESD} contains
a proof of Theorem~\ref{th: main ESD} based on the moment method, together with a proof of Propositions~\ref{prop: ESD}. 
Although Theorem~\ref{th: main ESD} is known, we prefer to include our argument
here since it provides a uniform combinatorial treatment of both the bulk and
the edges of the spectrum in an inhomogeneous setting.

\subsection*{Acknowledgments} The authors are grateful to Alexey Naumov for bringing to their attention the references \cite{GNT2014} and \cite{chin2023necessary}. We also thank Ramon van Handel for pointing out Remarks \ref{remark: ramon's remark} and \ref{rk: seginer}.

\section{The regime with no outliers}\label{sec: no outliers}
The goal of this section is to prove Theorem~\ref{th: outliers}. Let $\xi$ be a centered sub-Gaussian random variable with unit variance and $W_n=(w_{ij})_{1\leq i,j\leq n}$ be an $n\times n$ Wigner matrix whose entries on and above the diagonal are iid copies of $\xi$. 
Fix an $n\times n$ symmetric matrix $\Sigma_n=(\sigma_{ij})_{1\leq i,j\leq n}$ satisfying 
\begin{itemize}
    \item $(\sigma_{ij}^2)_{1\leq i,j\leq n}$ is doubly stochastic;
    \item $\sigma_n^*\sqrt{\log n}\to 0$, where $\sigma_n^*=\max_{1\leq i,j\leq n} \sigma_{ij}$.
\end{itemize}

Given $L=L_n$ to be specified later, define $W_n^{\leq L}$ (resp. $W_n^{> L}$) to be the matrix with entries $\big(w_{ij}\mathbf{1}_{|w_{ij}| \le L}\big)_{1\leq i,j\leq n}$ (resp. $\big(w_{ij}\mathbf{1}_{|w_{ij}| > L}\big)_{1\leq i,j\leq n}$). 
Correspondingly, define $X_n^{\leq L}:= \Sigma_n\circ W_n^{\leq L}$ (resp. $X_n^{>L}:= \Sigma_n\circ W_n^{>L}$) so that 
$$
X_n:= \Sigma_n\circ W_n= X_n^{\leq L}+X_n^{>L}.
$$
Given an integer $p$, the triangle inequality combined with the fact that the entries of $X_n$ are centered yields
\begin{align*}
    \left(\esp \tr(X_n^{2p})\right)^{1/(2p)} &\le \left(\esp \tr(X_n^{\le L}-\esp X_n^{\le L})^{2p}\right)^{1/(2p)}+\left(\esp \tr(X_n^{>L}-\esp X_n^{> L})^{2p}\right)^{1/(2p)}\\
    &=:\Gamma_1+\Gamma_2.
\end{align*}
In order to bound $\Gamma_1$, we make use of \cite[Theorem 4.8]{latala2018dimension} as the entries of $X_n^{\le L}$ are bounded, to obtain that 
$$
\Gamma_1\leq 2 \left(\sum_{j \in [n]}\left(\sum_{i \in [n]}\sigma_{ij}^2 \esp \big(w_{ij}\mathbf{1}_{|w_{ij}| \le L}-\esp w_{ij}\mathbf{1}_{|w_{ij}| \le L}\big)^2\right)^{p}\right)^{1/(2p)}+C n^{\frac{1}{p}}L\sigma_n^*\sqrt{p}, 
$$
for some universal constant $C$. Using that $(\sigma_{ij}^2)_{1\leq i,j\leq n}$ is doubly stochastic, $\xi$ has unit variance, and that the $w_{ij}$'s are iid, we deduce that 
\begin{equation}\label{eq: gamma1}
\Gamma_1\leq 2n^{\frac{1}{2p}} \sqrt{\var(\xi \mathbf{1}_{|\xi| \le L}) }+ C n^{\frac{1}{p}} L\sigma_n^*\sqrt{p}\leq n^{\frac{1}{p}} \Big( 2+ CL\sigma_n^*\sqrt{p}\Big). 
\end{equation}
To bound $\Gamma_2$, let $\tilde{X}_n^{>L}$ be an independent copy of $X_n^{>L}$. Since $X \mapsto \tr(X)^{2p}$ is convex, Jensen's Inequality implies that
\begin{align*}
    \Gamma_2=\left(\esp_{X_n} \tr\left(\esp_{\tilde{X}_n}\left(X_n^{>L}-\tilde{X}_n^{> L}\right)\right)^{2p}\right)^{1/(2p)} \le \left(\esp_{X_n,\tilde{X}_n} \tr(X_n^{>L}-\tilde{X}_n^{> L})^{2p}\right)^{1/(2p)}.
\end{align*}
By symmetry, the distribution of $X_n^{>L}-\tilde{X}_n^{> L}$ coincides with that of $R\circ (X_n^{>L}-\tilde{X}_n^{> L})$, where $R=(r_{ij})_{1\leq i,j\leq n}$ is an $n \times n$ symmetric matrix with independent Rademacher variables independent of $X_n^{>L}$ and $\tilde{X}_n^{>L}$. By the triangle inequality, it follows that
\begin{align*}
    \Gamma_2 \le 2(\esp \tr( R\circ X_n^{>L})^{2p})^{1/(2p)}.
\end{align*}
Expanding the trace, we get
\begin{align*}
    \esp \tr(R \circ X_n^{>L})^{2p}=\sum_{u \in [n]^{2p}}\sigma_{u_1u_2}\cdots \sigma_{u_{2p}u_1}\esp r_{u_1u_2}(X_n^{>L})_{u_1u_2}\cdots r_{u_{2p}u_1}(X_n^{>L})_{u_{2p}u_1}.
\end{align*}
Let $u_1\to \cdots \to u_{2p} \to u_1$ be a closed path. We define its shape $s(u)$ as the relabelling of its vertices with labels $1,2,\dots$ in order of appearance. Since the Rademacher variable is symmetric, for any path with non-zero contribution, all edges $(u_i,u_{i+1})$ must appear an even number of times on the path, and in this case, the corresponding contribution of the Rademacher variables is equal to  $1$.
Let $\mathcal{S}_{even}$ be the set of all even shapes. Note that the quantity
\begin{align*}
    T(s):=\esp (X_n^{>L})_{u_1u_2}\cdots (X_n^{>L})_{u_{2p}u_1}
\end{align*}
depends only on the shape $s(u)=s$ of the path $u$. In particular, we have
\begin{align*}
    \esp \tr(R \circ X_n^{>L})^{2p}=\sum_{s \in \mathcal{S}_{even}}T(s)\sum_{\substack{u \in [n]^{2p}\\ s(u)=s}}\sigma_{u_1u_2}\cdots \sigma_{u_{2p}u_1}.
\end{align*}
For a shape $s \in \mathcal{S}_{even}$, let $H_s=(V(s), E(s))$ be the graph generated by the vertices visited by the shape and edges given by  $(u_1,u_2),\ldots,(u_{2p},u_1)$. Denote $m(s)=|V(s)|$. Since $s$ is a path, $H_s$ is connected and we have $|E(s)| \ge m(s)-1$. Applying \cite[Lemma 2.5]{bandeira2016sharp} (see also \cite[Theorem 2.8]{latala2018dimension}) and using that $(\sigma_{ij}^2)_{1\leq i,j\leq n}$ is doubly stochastic, we have 
\begin{align*}
    \sum_{\substack{u \in [n]^{2p}\\ s(u)=s}}\sigma_{u_1u_2}\cdots \sigma_{u_{2p}u_1} \le n(\sigma_n^*)^{2p-2(m(s)-1)}.
\end{align*}
Therefore, we deduce that
\begin{equation}\label{eq: norm-upper truncated}
    \esp \tr(R \circ X_n^{>L})^{2p}\le n(\sigma_n^*)^{2p}\sum_{s \in \mathcal{S}_{even}}T(s)(\sigma_n^*)^{-2(m(s)-1)}.
\end{equation}
Now, let $k_e$ be the number of times the edge $e \in E(s)$ is traversed by the shape $s$. Then 
\begin{align*}
    T(s)=\prod_{e \in E(s)}\esp \xi^{k_e}\mathbf{1}_{|\xi| > L}.
\end{align*}
Using Cauchy--Schwartz inequality and that $\xi$ is sub-Gaussian (see \cite[Section 2.5]{vershynin2020high} for equivalent definitions of sub-Gaussian random variables), we can write
\begin{align*}
    \esp \xi^{k_e}\mathbf{1}_{|\xi| > L} \le (\esp \xi^{2k_e})^{1/2}\prob(|\xi|>L)^{1/2}\leq \tilde{C}^{k_e} e^{-cL^2}\esp g^{k_e}, 
\end{align*}
where $g \sim N(0,1)$ and $\tilde{C}$, $c$ are universal constants.

Since $\sum_{e \in E(s)}k_e=2p$ and $|E(s)| \ge m(s)-1$, putting together the above we deduce that
\begin{align*}
    T(s) \le \tilde{C}^{2p} e^{-c|E(s)|L^2}T_g(s)\leq \tilde{C}^{2p}e^{-c(m(s)-1)L^2}T_g(s),
\end{align*}
where $T_g(s)=\prod_{e \in E(s)}\esp g^{k_e}$. Replacing the above relations in \eqref{eq: norm-upper truncated}, we deduce that
\begin{align*}
     \esp \tr(R\circ X_n^{>L})^{2p}\le \tilde{C}^{2p}(\sigma_n^*)^{2p} n\sum_{s \in \mathcal{S}_{even}}T_g(s) b^{m(s)-1},
\end{align*}
where we denoted $b:=(\sigma_n^*)^{-2} e^{-cL^2}$.  
Setting $q=\lceil b\rceil +p$, using that $m(s) \le p+1$ and that
\begin{align*}
    \frac{q!}{(q-m(s))!}\ge q(q-m(s)+1)^{m(s)-1} \ge qb^{m(s)-1},
\end{align*}
we get 
\begin{align*}
     \esp \tr(R\circ X_n^{>L})^{2p}\le \tilde{C}^{2p} (\sigma_n^*)^{2p}\frac{n}{q}\sum_{s \in \mathcal{S}_{even}}T_g(s) \frac{q!}{(q-m(s))!}.
\end{align*}
Now note that for a standard $q \times q$ symmetric Gaussian matrix $G_q$, we readily have
\begin{align*}
    \esp \tr(G_q)^{2p}=\sum_{s \in \mathcal{S}_{even}}T_g(s)\frac{q!}{(q-m(s))!}.
\end{align*}
Thus, we deduce that 
\begin{align*}
    \esp \tr(R \circ X_n^{>L})^{2p}&\le \tilde{C}^{2p}(\sigma_n^*)^{2p}\frac{n}{q}\esp \tr(G_q)^{2p}\\ 
    &\leq \tilde{C}^{2p}(\sigma_n^*)^{2p} n \esp \norm{G_q}^{2p}\\ 
    &\leq \tilde{C}^{2p}(\sigma_n^*)^{2p} n(2\sqrt{q}+8\sqrt{p})^{2p},
\end{align*}
where we have used \cite[Lemma 2.2]{bandeira2016sharp} to bound the moments of the norm of a standard Gaussian matrix. 
Replacing $b$ in the above, we get that 
\begin{align}\label{eq: gamma2}
    \Gamma_2 
    \leq C'n^{1/(2p)}\left(e^{-cL^2/2}+\sigma_n^*\sqrt{p}\right),
\end{align}
for some universal constant $C'$. 
Combining \eqref{eq: gamma1} and \eqref{eq: gamma2}, and choosing $L^2=-2c^{-1}\log(\sigma_n^*\sqrt{p})$, we get that 
\begin{align}\label{ineq: final bound on X with choice L}
   \left(\esp \tr(X_n^{2p})\right)^{1/(2p)}  \le  n^{1/p} \left(2+C''\sigma_n^* \sqrt{p\log\Big(\frac{1}{\sigma_n^*\sqrt{p}}\Big)} \right),
\end{align}
for some universal constant $C''$. 
Now choosing 
\begin{equation}\label{eq: choice p}
p=\frac{\sqrt{\log n}}{\sigma_n^*},
\end{equation}
we deduce that 
\begin{align*}
\esp \Vert X_n\Vert \leq \left(\esp \Vert X_n\Vert^{2p}\right)^{\frac{1}{2p}}
\leq e^{\sigma_n^*\sqrt{\log n}} \left[2+C''\sqrt{\sigma_n^*\sqrt{\log n}} \sqrt{\log\Bigg(\frac{1}{\sqrt{\sigma_n^*\sqrt{\log n}}}}\Bigg) \right].  
\end{align*}
Since $\sigma_n^*\sqrt{\log n}\to 0$, we have that 
$$
\limsup_{n\to \infty}\esp \Vert X_n\Vert\leq 2.
$$
On the other hand, it follows from the almost sure convergence of the empirical spectral distribution (see Theorem~\ref{th: main ESD}) that almost surely,
$$
\liminf_{n\to \infty} \Vert X_n\Vert\geq 2,
$$
which, by Fatou's lemma, implies that 
$$
\liminf_{n\to \infty} \esp\Vert X_n\Vert\geq 2.
$$
This shows that $\esp\Vert X_n\Vert$ converges to $2$ whenever $\sigma_n^*\sqrt{\log n}\to 0$. 

To prove the almost sure convergence, note that for any fixed $\eta>2$, we have by Markov's inequality that 
\begin{align*}
    \prob(\Vert X_n\Vert>\eta) \le \frac{\esp \tr(X_n^{2p})}{\eta^{2p}}
    \leq n^2\left(\frac{2+C''\sqrt{\sigma_n^*\sqrt{\log n}} \sqrt{\log\big(\frac{1}{\sqrt{\sigma_n^*\sqrt{\log n}}}}\big)}{\eta}\right)^{2p}.
\end{align*}
Using the choice of $p$ in \eqref{eq: choice p} and that $\sigma_n^*\sqrt{\log n}\to 0$, it is easy to see that the above quantity is summable in $n$. A classical application of the Borel-Cantelli lemma finishes the proof. 

\begin{remark}[Communicated by Ramon van Handel]\label{remark: ramon's remark}
A compression argument along the lines of
\cite[Theorem~4.8]{latala2018dimension} can be used as the basis for an alternative proof of Theorem~\ref{th: outliers}. Such a compression argument would yield a comparison to a matrix $M$ with independent entries of the form $b|g|$,
where $b$ is a normalized Bernoulli variable (with properly chosen
probability of success) and $g$ is an independent standard Gaussian. 
By Talagrand's inequality \cite{talagrand95concentration} and Gaussian concentration, the largest eigenvalue of $M$
is subgaussian. Combining this with the Bai--Yin theorem \cite{bai1988necessary} rules out the presence of outliers in the asymptotic regime.
\end{remark}

\begin{remark}[Heavy-tailed distributions]
The above argument can be adapted to heavy-tailed distributions. Let $\xi$ be a centered Weibull distribution with shape parameter at most $2$, so that for some
$\beta \ge 1/2$ and any $p \ge 1$, we have
\begin{align*}
    &\norm{\xi}_p \le Cp^{\beta}.
\end{align*}
Let $\Sigma_n$ be a doubly stochastic matrix and $X_n=\Sigma_n \circ W_n$. Then, similarly to the proof of \cite[Theorem 4.4]{latala2018dimension} and the above computations, using $\sqrt{p} \le p^{\beta}$, we get
\begin{align*}
    \left(\esp \tr(X_n^{2p})\right)^{1/(2p)} \le n^{1/p}\left[2+C_\beta \left(e^{-c_\beta L^{1/\beta}}+\sigma_n^*L p^{\beta}\right)\right].
\end{align*}
Choosing
\begin{align*}
     L^{1/\beta}:=\frac{1}{c_\beta}\log\left(\frac{1}{\beta \sigma_n^*p^\beta}\right); \quad p:=\frac{\sqrt{\log n}}{(\sigma_n^*)^{1/(2\beta)}},
\end{align*}
we get
\begin{align*}
    \left(\esp \tr(X_n^{2p})\right)^{1/(2p)} \le e^{(\sigma_n^*\log^\beta n)^{1/(2\beta)}}\left(2+C_\beta\sqrt{\sigma_n^*\log^\beta n}\log^\beta\left(\frac{1}{\sqrt{\sigma_n^*\log^\beta n}}\right)\right).
\end{align*}
Therefore, whenever $\sigma_n^*\log^\beta n \to 0$, $X_n$ has no outliers almost surely and in expectation. This improves upon \cite[Theorem 4.4]{lata2022norms} 
since the cited theorem does not provide
the correct leading constant $2$ for the spectral norm estimate.
\end{remark}

\section{Outliers in the subcritical regime}\label{sec: outliers}

The goal of this section is to prove
Theorem~\ref{th: presence outliers}.
Note that since the random variable $\xi$ is centered and non-constant, it satisfies
\begin{equation}\label{eq: xi anticoncentrate}
   \sup_{u\in \mathbb{R}} \prob\left( \vert \xi- u\vert \geq \rho\right) \geq \rho,\quad \prob\left(\xi\geq \rho\right)\geq \rho, \quad \text{and}\quad \prob\left(\xi\leq -\rho\right)\geq \rho,
\end{equation}
for some $\rho\in (0,1)$ depending only on the distribution of $\xi$. 

Let $n$ be a large integer, and let $d=d(n)=O(\log n)$ and $d=\omega(1)$.
To simplify the exposition, we will assume that $n/(d+1)$ is an integer; we note that our construction below can be easily adapted to cover all admissible choices of $d=d(n)$.
Consider a $d$-regular graph $G_n$ on $n$ vertices whose adjacency matrix $A_n$ is block diagonal with $n/(d+1)$ blocks of all ones (excluding the entries on the main diagonal). The graph $G_n$ is thus a disjoint union of $n/(d+1)$ cliques of size $d+1$ each. 
For every integer $k$, we denote by $\tilde W_k$ the $k\times k$ symmetric matrix whose entries above the main diagonal are iid copies of $\xi$, and the main diagonal is zero. 
Finally, we define $X_n:=\frac{1}{\sqrt{d}} A_n\circ \tilde W_n$, so that
$X_n$ is block diagonal with $n/(d+1)$ iid blocks, where each block
is equidistributed with $\frac{1}{\sqrt{d}} \tilde W_{d+1}$. We denote these blocks by $W_{d+1}^{(i)}$, $i=1,\ldots, n/(d+1)$. 
To prove Proposition~\ref{th: presence outliers}, we will show that $\liminf_{n\to \infty}\Vert X_n\Vert$ is almost surely bounded away from $2$. 
We will need the following lemma:
\begin{lemma}\label{lem: lower bound norm}
    For every $\varepsilon>0$ there is $\delta>0$
    depending on $\varepsilon$ and the distribution of $\xi$ with the following property. Let $k$ be a sufficiently large integer, and let $\tilde W_{k+1}$
    be the matrix defined as above. Then
       \begin{align*}
         \prob\left(\norm{\tilde W_{k+1}} \ge (2+\delta)\sqrt{k}\right) \ge \exp(-\varepsilon k).
    \end{align*}
\end{lemma}

We provide the proof of the lemma at the end of the subsection.
At this point, let $C>0$ be the constant such that $d=d(n)\leq C\log n$
for all sufficiently large $n$. We choose $\varepsilon:=\frac{1}{2C}$.
Following Lemma~\ref{lem: lower bound norm},
there is $\delta\in (0,1)$ depending on $C$
and the distribution of $\xi$ such that
\begin{align*}
    \prob(\Vert X_n\Vert \ge 2+\delta)&=\prob\left(\max_{i=1,\ldots,n/(d+1)}\Vert W^{(i)}_{d+1}\Vert \ge (2+\delta)\sqrt{d}\right)\\
    &= 1- \prob\left(\Vert \tilde W_{d+1}\Vert <  (2+\delta)\sqrt{d}\right)^{\frac{n}{d+1}}\\
    &\geq 1-\left(1-e^{-\varepsilon d}\right)^{\frac{n}{d+1}}\\
    &\geq 1-\left(1-\frac{1}{\sqrt{n}}\right)^{\frac{n}{d+1}}\\
    &\geq 1-e^{-n^{1/4}},
\end{align*}
for all large $n$.
A direct application of the Borel--Cantelli lemma implies the result.

\bigskip
\begin{remark}\label{rk: seginer}
 A weaker version of Lemma~\ref{lem: lower bound norm} follows easily by an adaptation of an argument in \cite{seginer2000expected} (see the proof of Theorem~3.2 there). Indeed, one can write 
 $$
\prob\left(\norm{\tilde W_{k+1}} \ge 3 \sqrt{k}\right)
\geq 
\prob\left(\xi\geq \rho\right)^{\frac{9}{\rho^2} k}. 
 $$
This follows since the norm of  $\tilde W_{k+1}$ is bounded below by the norm of its $(\frac{3}{\rho} \sqrt{k})\times (\frac{3}{\rho} \sqrt{k})$ submatrix. The latter is larger than $3 \sqrt{k}$ if all entries are larger than $\rho$. This trivial observation gives a version of Lemma~\ref{lem: lower bound norm} for some $\varepsilon$ (rather than for any $\varepsilon$). One can easily check that using this, we have $\frac{\Vert X_n\Vert}{2}\to \infty$ almost surely whenever $d/\log n\to 0$. In this regime, the above argument indicates that the presence of outliers is caused by the emergence of a submatrix of size $O(\sqrt{\log n})$ with large entries of equal signs. 
  
The extra quantification (in terms of $\varepsilon$)  present in Lemma~\ref{lem: lower bound norm} is needed to treat the case of $d$ of order $\log n$. 
On the other hand, we note that stronger and sharper
statements than Lemma~\ref{lem: lower bound norm} 
follow from the large deviation principle for the largest eigenvalue of Wigner matrices \cite{augeri2021largedeviation,guionnet2020rademacher,CDG2023}; however, those results only cover particular cases of sub-Gaussian random variables.
\end{remark}

\begin{proof}[Proof of Lemma~\ref{lem: lower bound norm}]
We start the proof with the following observation.
For every choice of parameters $\varepsilon,\beta>0$
there is $\gamma>0$ depending on $\beta,\varepsilon$ and the distribution
of $\xi$ with the following property. Assuming $k$
is sufficiently large, letting $v$ be a non-random
unit vector in $\R^k$ with
$$
\big|\big\{i\leq k:\;|v_i|\geq \beta/\sqrt{k} \big\}\big|
\geq \beta\,k,
$$
and taking $\xi_1,\dots,\xi_k$ to be iid copies of $\xi$,
we have
$$
\prob\bigg(\Big|\sum_{i=1}^k v_i\xi_i\Big|\geq \gamma\,\sqrt{k}
\bigg)\geq \exp\big(-\varepsilon k/2\big).
$$
Indeed, assuming $k$ is large enough, we take
$\ell$ to be the largest integer bounded above by $\beta\,k$ and
such that $\rho^\ell\geq \exp\big(-\varepsilon k/4\big)$,
where $\rho$ is taken from \eqref{eq: xi anticoncentrate}
(observe that $\ell$ is of order $k$).
Further, let $I\subset[k]$ be a non-random subset of cardinality $\ell$
such that $|v_i|\geq \beta/\sqrt{k}$ for every $i\in I$.
In view of \eqref{eq: xi anticoncentrate} and the choice of $\ell$, with probability
at least $\exp\big(-\varepsilon k/4\big)$ we have
$$
\sum_{i\in I} v_i\xi_i\geq \frac{\beta}{\sqrt{k}}\cdot\rho\ell.
$$
On the other hand, since the variance of
$
\sum_{i\in [k]\setminus I} v_i\xi_i
$
is less than one and in view of Markov inequality,
with a probability of at least $3/4$, we have
$$
\bigg|\sum_{i\in [k]\setminus I} v_i\xi_i\bigg|\leq 2.
$$
Combining the two estimates, we get that
$$
\sum_{i=1}^k v_i\xi_i\geq \frac{\beta}{\sqrt{k}}\cdot\rho\ell-2
$$
with probability at least $\frac{3}{4}\exp\big(-\varepsilon k/4\big)
\geq \exp\big(-\varepsilon k/2\big)$, and the claim follows.

\medskip

For convenience, we will use the compact notation $M:=\tilde W_{k+1}$,
and we let $M'$ be the top left $k\times k$ principal submatrix of $M$.
Let $\alpha>0$ be a small parameter (depending on $\varepsilon$ and
the distribution of $\xi$) which will be determined later.
The standard covering arguments imply that, as long as $k$
is sufficiently large, with a probability of at least $0.999$ 
{\it every} eigenvector of $M'$ is {\it incompressible} (see \cite{rudelson2016no}). On the other hand, as a consequence of the Wigner
semicircle law (Theorem~\ref{th: main ESD}), with a
probability of at least 
$0.999$ there is an eigenvector $v$ of $M'$
with $\|M'v\|_2\geq (2-\alpha)\sqrt{k}$.
To summarize, we can define an $M'$--measurable random unit vector $v$
satisfying
$$
\prob\big(\|M'v\|_2\geq (2-\alpha)\sqrt{k},\;\;
\big|\big\{i\leq k:\;|v_i|\geq \beta/\sqrt{k} \big\}\big|
\geq \beta\,k\big)\geq 0.99,
$$
where $\beta$ depends on the distribution of $\xi$ but not on $\alpha$.
Let $z$ be a vector in $\R^{k+1}$
obtained from $v$ by adding zero coordinate.
Our goal is to estimate the probability that $\|Mz\|_2$
is bounded away from two.
Denote the $(k+1)$-st row of $M$ by $\row_{k+1}$.
We have
    \begin{align}\label{ineq: norm W_d with norm W_d-1}
        \norm{M}^2 \ge \norm{Mz}^2_2
        =\norm{M'v}_2^2+\langle \row_{k+1},z\rangle^2.
    \end{align}
In view of the definition of $v$, we have
\begin{equation}\label{aljkhbreoguygor}
\norm{M'v}^2_2\geq (2-\alpha)^2\,k
\end{equation}
with probability at least $0.99$.
On the other hand, conditioned on any realization of $M'$
such that \eqref{aljkhbreoguygor} holds and
$$
\big|\big\{i\leq k:\;|v_i|\geq \beta/\sqrt{k} \big\}\big|
\geq \beta\,k,
$$
we obtain from the observation at the beginning of the proof
that
$$
\prob_{\row_{k+1}}\big(|\langle \row_{k+1},z\rangle|\geq \gamma\sqrt{k}\big)
\geq\exp\big(-\varepsilon k/2\big),
$$
for some $\gamma$ depending on $\beta$, $\varepsilon$, and the distribution of $\xi$ (but not on $\alpha$).
It remains to choose $\alpha$ so that
$(2-\alpha)^2+\gamma^2>4$, and the proof is complete.

\end{proof}

\newpage

\appendix
\section{The limiting Empirical Spectral Distribution:
A Moment Method Approach}\label{sec: universal ESD}

The goal of this appendix is to provide a proof of Theorem~\ref{th: main ESD}
based on the moment method.
We will assume that the variance profile $\Sigma_n$ satisfies $\sigma_n^*\to 0$. 
By a standard truncation argument (see \cite[Section 2.4.1]{tao2012randommatrix}),
it is sufficient to deal with matrices with bounded entries and zero diagonal.  
More precisely, we let $(L_n)_{n\geq 1}$ be a sequence of numbers
defined by
\begin{equation}\label{eq: choice L}
    L_n:=\min\left\{\frac{1}{\sqrt{\sigma_n^*}},\log^cn\right\}
\end{equation}
(so that $L_n\to\infty$),
and from now on we will
assume that the random matrices
$W_n=(w_{ij})_{1\leq i,j\leq n}$ have i.i.d centered entries
with unit variances and are {\it uniformly bounded by $L_n$}.
Here, $c \ge 1$ is a fixed large constant. 

The convergence of the ESD of $X_n$ is characterized by the convergence of its moments. More precisely, we aim to show that for every integer $k\geq 1$, the sequence $\big(\frac{1}{n}\tr(X_n^k)\big)_{n\geq 1}$ converges almost surely and in expectation to the corresponding $k$th moment of the semi-circle distribution. 

Given an integer vector $u \in [n]^k$, we will treat it
as a closed path $u=u_1 \to \cdots \to u_k \to u_1$ in the complete graph over $[n]$. The path $u$ generates a (connected) subgraph $H_u$ whose vertices are $\{u_1,\ldots,u_k\}$ and the edges are $\{(u_l,u_{l+1}):l \in [k]\}$, where $u_{k+1}=u_1$. With some abuse of notation, we will fix a labeling of the vertices of $H_u$, namely, $V(H_u)=\{1,\ldots, |V(H_u)|\}$. Let $q_u(e)$ be the number of times the edge $e \in E(H_u)$ is traversed by the path $u$, so that $\sum_e q_u(e)=k$. 
Finally, for every $u\in [n]^k$, we denote 
$$
\sigma_u:= \sigma_{u_1u_2}\ldots \sigma_{u_ku_1}\quad\text{and}\quad w_u=w_{u_1u_2}\ldots w_{u_ku_1}. 
$$

\subsection{The convergence in expectation} 
To prove the convergence in expectation, we write 
\begin{align*}
\frac{1}{n}\esp \tr (X^k_n)=\frac{1}{n}\sum_{u \in [n]^k}\sigma_u\esp w_u.
\end{align*}
Since $W_n$ is centered, the case $k=1$ is trivial. Thus, we suppose in the sequel that $k\geq 2$. 
Note that for a given $u\in [n]^k$, each edge $e \in E(H_u)$ must be traversed at least twice; otherwise, their contribution is zero as the variables are independent and centered. This implies that $q_u(e)\geq 2$ for every edge $e \in E(H_u)$ and therefore that $|E(H_u)| \le k/2$. 
Since $H_u$ is a connected graph, we have 
$$
|V(H_u)| \le |E(H_u)|+1\le k/2+1.
$$
We will write $q_u\geq 2$ (resp. $q_u=2$) to encode the fact that $q_u(e)\geq 2$ (resp. $q_u=2$) for every edge $e \in E(H_u)$. 
Note that when $q_u\geq 2$, we have 
\begin{align}\label{ineq: bound summands}
    \esp w_u \leq L_n^{k-2|E(H_u)|}\leq L_n^{k-2(|V(H_u)|-1)},
\end{align}
where we used that the entries of $W_n$ are of unit variance and uniformly bounded by $L_n$.

Applying \cite[Lemma 2.5]{bandeira2016sharp} (see also \cite[Theorem 2.8]{latala2018dimension}) and using that $(\sigma_{ij}^2)_{1\leq i,j\leq n}$ is doubly stochastic, we can write for every $m\leq k/2+1$ that
\begin{align}\label{ineq:vertices low}
\frac{1}{n}\sum_{\substack{u \in [n]^k\\ q_u \ge 2, |V(H_u)|= m}}\sigma_u\esp w_u   \le (L_n\sigma_n^*)^{k-2(m-1)}= (\sqrt{\sigma_n^*})^{k-2(m-1)},
\end{align}
by the choice of $L_n$ in \eqref{eq: choice L}. Note that whenever $m \le k/2$, we get that $k-2(m-1) \ge 2$ and the right-hand side of \eqref{ineq:vertices low} goes to $0$. Therefore, we can restrict to the case where $|V(H_u)|=k/2+1$, that is, $H_u$ is a tree with $|E(H_u)|=k/2$ edges. Moreover, $q_u(e)=2$ for all edges (otherwise $|E(H_u)|<k/2$). Hence, we get that
\begin{align*}
\lim_{n \to \infty }\frac{1}{n}\esp \tr (X^k_n)=\lim_{n \to \infty}\frac{1}{n}\sum_{\substack{u \in [n]^k\\ q_u=2;H_u\text{ is a tree}}}\sigma_u\esp w_u.
\end{align*}
In this case, $\esp w_u=1$, so that 
\begin{align*}
\lim_{n \to \infty }\frac{1}{n}\esp \tr (X^k_n)=\lim_{n \to \infty}\frac{1}{n}\sum_{\substack{u \in [n]^k\\q_u=2;H_u\text{ is a tree}}}\sigma_u.  
\end{align*}
Consider now the set $\mathcal{G}_k$ of all pairs $(G,t)$ where $G$ is a tree over $[k/2+1]$ with a closed walk $t=t_1\to \cdots \to t_k \to t_1$ that traverses each edge exactly twice. Then
\begin{align*}
\lim_{n \to \infty }\frac{1}{n}\esp \tr (X^k_n)=\lim_{n \to \infty}\sum_{(G,t) \in \mathcal{G}_k}\frac{1}{n}\sum_{\substack{u \in [n]^k\\ (H_u,u)=(G,t)}}\sigma_u.
\end{align*}
As $G$ is a tree and each edge is traversed twice, we can sum up over leaves. Using double stochasticity and induction over leaves, we have
\begin{align*}
    \frac{1}{n}\sum_{\substack{u \in [n]^k\\ (H_u,u)=(G,t)}}\sigma_u=\frac{1}{n}\sum_{i,j \in [n]}\sigma_{ij}^2=1.
\end{align*}
This also follows by a similar argument done in \cite[Lemma 2.5]{bandeira2016sharp}. By \cite[Exercise 4.4.1]{kemp2013introduction}, it is known that $|\mathcal{G}_k|=|NC_2(k)|$ which is precisely the $k$th moment of the semi-circle distribution. Hence, the result follows.

\subsection{Tail bound and almost-sure convergence}
By the Borel-Cantelli lemma, the following proposition immediately implies almost sure convergence of the ESD of $X_n$. 

\begin{proposition}\label{proposition: trace-tails} 
    For all $k \ge 1$, $p \ge 2$, and $\eps \in (0,1/2)$, we have
    \[
        \mathbb P\left[\frac{1}{n}\left|\tr(X_n^k) - \mathbb E[\tr(X_n^k)]\right| > n^{-1/2+\eps}\right] \le \frac{2}{n^p},
    \]
    for all $n$ sufficiently large depending on $k,p$ and $\eps$.
\end{proposition}

\begin{proof}
    Let $\FF_t$ be the $\sigma$-algebra generated by revealing the first $t$ of the entries on or above the diagonal of $X$ under some arbitrary ordering. Denote $\E[t]{\cdot}:= \E{\cdot|\FF_t}$. In this notation,
    \begin{align*}
        \tr(X^k_n) - \E{\tr(X^k_n)} &= \sum_{t = 1}^{n(n+1)/2} \E[t]{\tr(X^k_n)} - \E[t-1]{\tr(X^k_n)} \\
        & =: \sum_{t=1}^{n(n+1)/2} \Delta_t \,.
    \end{align*}
    The key estimate is the following ``bounded difference'' estimate for the martingale increments $\Delta_t$. 
    For all $t \in [n(n+1)/2]$, 
        \begin{align}\label{ineq: trace-tails-delta_t}
            \PP{|\Delta_t|^2 > n^{\eps}\sigma_t^2} \le \frac{1}{n^{p+2}}\,,
        \end{align}
    Once established, this readily implies Proposition \ref{proposition: trace-tails}. By union bound:
    \[
        \PP{\exists\, t \in [n(n+1)/2]\,:\,|\Delta_t| > n^{\eps/2}\sigma_t} \le \frac{1}{n^p}
    \]
    Thus, by the Azuma--Hoeffding inequality for martingales with bounded increments, since $\pa{\sigma_{ij}^2}_{ij}$ is doubly stochastic:
    \begin{align*}
        \PP{\ba{\tr(X^k) - \E{\tr (X^k)} } > y} &\le \frac{1}{n^p} + \PP{\ba{\tr(X^k) - \E{\tr (X^k)} }> y,\,\Delta_t^2 \le n^{\eps}\sigma_t^2 ~\forall t} \\
        &\le \frac{1}{n^p} + \cE{-\frac{y^2}{2 n^{1+\eps}}}\,. 
    \end{align*}

    Taking $y = n^{\frac{1}{2}+\eps}$ concludes the proof. We turn towards proving the proposition. \\

    Our strategy is to compute the $p$'th moment of $\Delta_t$ and then apply Markov's inequality, where $p$ is a large even integer independent of $n$. Fix $t$ to be the edge $(i,j)$ where $(i,j)$ is the $t$'th entry of $A$. Let $p$ be an even positive integer. We will say that $t \in u$ if $(u_l,u_{l+1})=t$ or $(u_{l+1},u_l)=t$ for some $l \in [k]$. Noting that $\E[t]{w_u} = \E[t-1]{w_u}$ if $t \not\in u$, we obtain
    \begin{align*}
        \E{\Delta_t^p} &= \E{\pa{\sum_{u \in [n]^k} \sigma_u\left(\E[t]{w_u} - \E[t-1]{w_u}\right)}^p} = \E{\pa{\sum_{\substack{u \in [n]^k\\ t \in u}} \sigma_u\left(\E[t]{w_u} - \E[t-1]{w_u}\right)}^p}\,.
    \end{align*}
    That is,
    \begin{align*}
        \E{\Delta_t^p}=\sum_{\substack{\mathbf{u}=(u^{(1)},\ldots,u^{(p)}) \in [n]^{kp}\\ t \in u^{(l)},\forall l \in [p]}}\sigma_{u^{(1)}}\cdots \sigma_{u^{(p)}}\esp \left(\prod_{l \in [p]}\left\{\E[t]{w_{u^{(l)}}}-\E[t-1]{w_{u^{(l)}}}\right\}\right).
    \end{align*}    
    Denote $H=H_{\mathbf{u}}$ the subgraph generated by $\mathbf{u}$. Let $q(e)=q_{\mathbf{u}}(e)$ be the number of times an edge $e \in E(H)$ is traversed in $\mathbf{u}$. As before, we only consider $\mathbf{u}$ such that $q(e)\ge 2$ for all $e \in E(H)$. Since all paths $u^{(l)}$ contain the edge $t \in E(H)$, we have that $H$ is connected and $q(t) \ge p$. We first bound the contribution of $W_n$.
    \begin{align*}
        \left|\esp \left(\prod_{l \in [p]}\left\{\E[t]{w_{u^{(l)}}}-\E[t-1]{w_{u^{(l)}}}\right\}\right)\right| &\le C_p\max_{r \in \{t,t-1\}^p}\esp \left(\prod_{l \in [p]}\E[r_l]{w_{u^{(l)}}}\right).
    \end{align*}
    Since $|w_{ij}| \le L_n$ and $w_{ij}$ has unit variance, recalling \eqref{ineq: bound summands}, 
    \begin{align}\label{ineq: bound qk moment}
        \left|\esp \left(\prod_{l \in [p]}\left\{\E[t]{w_{u^{(l)}}}-\E[t-1]{w_{u^{(l)}}}\right\}\right)\right| &\le C_pL_n^{pk-2(|V(H)|-1)}\,.
    \end{align}
    Thus, it suffices to control
    \begin{align*}
        \sum_{\substack{\mathbf{u} \in [n]^{pk}:\\(H_{\mathbf{u}},q_{\mathbf{u}})=(H,q),\\ t \in u^{(l)}, \forall l \in [p]}}\sigma_{u^{(1)}}\cdots \sigma_{u^{(q)}}\le \sum_{\substack{u \in [n]^{V(H)}\\ (u_i,u_j)=(i,j)}} \prod_{e=(a,b) \in E(H)}\sigma_{u_au_b}^{q(e)}.
    \end{align*}
    In order to use the results from \cite{latala2018dimension}, we will write this quantity in a better way. First, as $\sigma_{uv} \le \sigma_n^*$ for all $u,v \in [n]$, for any spanning tree $T=(V(T),E(T))$ of $H$ with $t \in E(T)$, we have
    \begin{align*}
        \sum_{\substack{u \in [n]^{V(H)}\\ (u_i,u_j)=(i,j)}} \prod_{e=(a,b) \in E(H)}\sigma_{u_au_b}^{q(e)} \le \prod_{e \in E(H)\setminus E(T)}(\sigma_n^*)^{q(e)}\sum_{\substack{u \in [n]^{V(T)}\\ (u_i,u_j)=(i,j)}} \prod_{e=(a,b) \in E(T)}\sigma_{u_au_b}^{q(e)}.
    \end{align*}
    Now define the following collection of (symmetric) matrices $(b^{(e)})_{e \in E(T)}$. 
    \begin{enumerate}
        \item If $e=t$, then $b^{(t)}_{ij}=b^{(t)}_{ji}:=\sigma_{ij}^{q(t)}$ and $0$ otherwise.
        \item If $e$ is incident to $i$, set $b^{(e)}_{ui}=b^{(e)}_{iu}=\sigma^{q(e)}_{ui}$ for all $u \in [n]$ and zero otherwise.
        \item If $e$ is incident to $j$, set $b^{(e)}_{uj}=b^{(e)}_{ju}=\sigma^{q(e)}_{uj}$ for all $u \in [n]$ and zero otherwise.
        \item If $e$ is not incident to neither $j$ nor $i$, set $b^{(e)}_{uv}=\sigma_{uv}^{q(e)}$ for all $u,v \in [n]$.
    \end{enumerate}
    Then it is immediate to see that
    \begin{align*}
        \sum_{\substack{u \in [n]^{V(T)}\\ (u_i,u_j)=(i,j)}} \prod_{e=(a,b) \in E(T)}\sigma_{u_au_b}^{q(e)} &\le \sum_{u \in [n]^{V(T)}} \prod_{e=(a,b) \in E(T)}b^{(e)}_{u_au_b}\\
        &=:W^b(T).
    \end{align*}
    We can readily use \cite[Lemma 2.10]{latala2018dimension} for $W^b(T)$ and deduce that
    \begin{align*}
        W^b(T) \le \prod_{e \in E(T)}\left\{\sum_{u \in [n]}\left(\sum_{v \in [n]}b^{(e)}_{uv}\right)^{p_e}\right\}^{1/p_e},
    \end{align*}
    where $(p_e)_{e \in E(T)}$ is any collection of conjugate exponents (namely $p_e \ge 1$ for all $e$, and $\sum_e 1/p_e = 1$). We set $p_e=\infty$ for all $e \ne t$ and $p_t=1$, so
    \begin{align*}
        W^b(T) \le \sigma_{t}^{q(t)}\prod_{e \in E(T)\setminus\{t\}}\max_{u \in [n]}\sum_{v \in [n]}b^{(e)}_{uv}. 
    \end{align*}
    Since $q(e) \ge 2$ for all $e \in E(H)$ and $\sigma_{uv} \le \sigma_n^*$, by double stochasticity, we get
    \begin{align*}
        W^b(T) \le \sigma_{t}^{q(t)} \prod_{e \in E(T)\setminus \{t\}}(\sigma_n^*)^{q(e)-2}.
    \end{align*}
    We deduce that
    \begin{align*}
        \sum_{\substack{\mathbf{u} \in [n]^{pk}:\\(H_{\mathbf{u}},q_{\mathbf{u}})=(H,q),\\ t \in u^{(l)} \forall l \in [p]}}\sigma_{u^{(1)}}\cdots \sigma_{u^{(q)}} \le \sigma_{t}^{q(t)} \prod_{e \in E(T)\setminus \{t\}}(\sigma_n^*)^{q(e)-2}\prod_{e \in E(H)\setminus E(T)}(\sigma_n^*)^{q(e)}.
    \end{align*}
    As $\sum_{e \in E(H)}q(e)=pk$, $q(t) \ge p$ and $|E(T)|=|V(H)|-1$, we get
    \begin{align}\label{ineq: sigma p-moment of delta_t}
        \sum_{\substack{\mathbf{u} \in [n]^{pk}:\\(H_{\mathbf{u}},q_{\mathbf{u}})=(H,q),\\ t \in u^{(l)}, \forall l \in [p]}}\sigma_{u^{(1)}}\cdots \sigma_{u^{(q)}} \le \sigma_{t}^p (\sigma_n^*)^{pk-p-2(|V(H)|-2)}.
    \end{align}
    Combining \eqref{ineq: bound qk moment} and \eqref{ineq: sigma p-moment of delta_t}, we get
    \begin{align*}
        \esp \Delta_t^p \le C_p(L_n\sigma_{t})^{p}\sum_{(H,q)} (L_n\sigma_n^*)^{pk-p-2(|V(H)|-1)} (\sigma_n^*)^2.
    \end{align*}
    To conclude the argument, since $q(t) \ge p$ and $q(e) \ge 2$ for all $e \in E(H)$, we have
    \begin{align*}
        |E(H)| \le \frac{pk-p}{2}+1.
    \end{align*}
    In particular,
    \begin{align*}
        |V(H)| \le |E(H)|+1 \le \frac{pk-p}{2}+2.
    \end{align*}    
    By the choice of $L_n$ in \eqref{eq: choice L}, we get $L_n \sigma_n^* \le 1$. The summation over the tuple $(H,q)$ has a finite number of terms depending only on $p$ and $k$. By replacing the value of $L_n$ given in \eqref{eq: choice L}, we deduce that
    \begin{align*}
        \esp \Delta_t^p &\le C_{p,k}\sigma_t^p L_n^{p-2}
    \end{align*}
    Markov's inequality implies that
    \begin{align}\label{ineq: tail bound delta_t final inequality}
        \prob\left(|\Delta_t| \ge s\sigma_{t}\right) \le C_{p,k}\frac{L_n^{p-2}}{s^p}.
    \end{align}
    Since $L_n \le \log^{c}n$, applying \eqref{ineq: tail bound delta_t final inequality} for $s=n^{\eps}$ and $p'=(p+3)/\eps$, we get
    \begin{align*}
        \prob\left(|\Delta_t| \ge n^\eps\sigma_{t}\right) \le C_{p',k}\frac{\log^{(p+3)c/\eps}n}{n^{p+3}} \le \frac{1}{n^{p+2}},
    \end{align*}
    for all $n\ge n(\eps,k,p)$, where $n(\eps,k,p)$ is a large constant depending on $\eps,k,p$. This implies \eqref{ineq: trace-tails-delta_t}, and Proposition \ref{proposition: trace-tails} follows.
\end{proof}

\subsection{The non-universal regime for the limiting ESD}

The goal of this subsection is to show that the sparsity assumption $\sigma_n^* \to 0$ is necessary in Theorem \ref{th: main ESD} for convergence to the semicircle law. This is formalized by the following proposition, and the proof is provided for completeness.
\begin{proposition}\label{prop: ESD}
Let $d\geq 2$ be constant with respect to $n$, and assume $n\to\infty$
with $nd$ even.
    \begin{itemize}
        \item (Distributional Non-universality) Let $W_n$ and $W_n'$ be independent symmetric matrices, with entries drawn as independent copies of bounded random variables $\xi$ and $\xi'$ respectively, both of which are symmetrically distributed and have unit variance. If the laws of $\xi$ and $\xi'$ are not equal, then there is a sequence of 
        $d$-regular graphs with adjacency matrices $A_n$, such that
        the matrices $X_n := \frac{1}{\sqrt{d}}A_n \circ W_n$ and $X_n' := \frac{1}{\sqrt{d}}A_n \circ W_n'$ have different limiting ESD almost surely. 
        
        \item (Structural Non-universality) Let $W_n$ be a symmetric matrix with independent Rademacher entries. There exist two sequences of $d$-regular graphs with adjacency matrices $A_n$ and $A_n'$ so that $X_n := \frac{1}{\sqrt{d}} A_n \circ W_n$ and $X_n' := \frac{1}{\sqrt{d}} A_n' \circ W_n$ have different limiting ESD almost surely.
    \end{itemize}
\end{proposition}

\begin{remark}
We note that the distributional non-universality in Proposition~\ref{prop: ESD} was originally proved in \cite{goldmakher2014weightedgraphs} for graphs $G_n$ that are locally tree-like in the Benjamini and Schramm topology \cite{benjamini2011recurrence}. The authors of \cite{goldmakher2014weightedgraphs}, however, focus on the "eigendistributions" of the limit ESD of $A_{G_n}\circ W_n$ as a function of $\xi$.
\end{remark}

\begin{proof}
Given a centered random variable $\xi$ with unit variance, we denote by $W_n^\xi$ the $n\times n$ Wigner matrix whose entries on and above the main diagonal are iid copies of $\xi$. 
We fix an integer $d\geq 2$ independent of $n$, and for every $d$-regular graph $G_n$ on $n$ vertices, we denote by $A_{G_n}$ its adjacency matrix. 
To prove distributional non-universality, we will construct a sequence of $d$-regular graphs $G_n$ and show that the limiting spectral distribution of $\frac{1}{\sqrt{d}} A_{G_n}\circ W_n^{\xi}$ depends on $\xi$. To show the dependence on the graph structure, we will set $\xi$ to be a Rademacher variable and construct two sequences of $d$-regular graphs $G_n$ and $G_n'$ such that the limiting spectral distributions of $\frac{1}{\sqrt{d}} A_{G_n}\circ W_n^{\xi}$ and $\frac{1}{\sqrt{d}} A_{G_n'}\circ W_n^{\xi}$ are different.

Regarding distributional non-universality, let $C_d$ be the clique in $d+1$ vertices, i.e. a complete graph on $d+1$ vertices excluding self-loops. 
To simplify the exposition, we will
suppose that $n$ is a multiple of $d+1$ and let $G_n$ be a $d$-regular graph on $n$ vertices given by the union of $n/(d+1)$ cliques $C_d$. 
Denote the entries of $A_{G_n}$ by $(a_{ij})_{1\leq i,j\leq n}$, the entries of $W_n^{\xi}$ by $(w_{ij}^{(\xi)})_{1\leq i,j\leq n}$, and set $X_n= \frac{1}{\sqrt{d}} A_{G_n}\circ W_n^{\xi}$.
We can write 
\begin{align*}
    \frac{1}{n}\esp \tr(X_n^{2k})=\frac{1}{d^k}\frac{1}{n}\sum_{u \in [n]^{2k}}a_{u_1u_2}\cdots a_{u_{2k}u_1}\esp w_{u_1u_2}^{(\xi)}\cdots w_{u_{2k}u_1}^{(\xi)},
\end{align*}
for every integer $k$. Fix a vertex $o \in C_d$. For a graph $H$
and a vertex $u_1$ in $H$, let $\mathcal{P}(u_1,H,2k)$ be the set of paths in $H$ starting and ending at $u_1$ and of length $2k$. Then, from the above
\begin{align*}
    \frac{1}{n}\esp \tr(X_n^{2k})=\frac{1}{nd^k}\sum_{u_1 \in [n]}\sum_{u \in \mathcal{P}(u_1,C_d(u_1),2k)}\esp w_{u_1u_2}^{(\xi)}\cdots w_{u_{2k}u_1}^{(\xi)},
\end{align*}
where by $C_d(u_1)$ we denoted the $(d+1)$--clique in $G_n$ containing $u_1$.
Since $G_n$ is a disjoint union of copies of $C_d$ and the variables in each clique are iid, we deduce that
\begin{equation}\label{eq: ESD clique}
    \frac{1}{n}\esp \tr(X_n^{2k})=\frac{1}{d^k}\sum_{u \in \mathcal{P}(o,C_d,2k)}\esp w_{u_1u_2}^{(\xi)}\cdots w_{u_{2k}u_1}^{(\xi)}=:m((C_d,o),2k,\xi).
\end{equation}
It is immediate to see that the sequence
$(m((C_d,o),2k,\xi))_{k\geq 1}$
uniquely determines all the even moments of $\xi$
(more specifically, if $k_{\min}$ is the smallest integer such that
two distinct normalized symmetric distributions $\xi$ and $\xi'$
have different $2k_{\min}$--th moments then necessarily
$m((C_d,o),2k_{\min},\xi)\neq m((C_d,o),2k_{\min},\xi')$).
This completes the proof of distributional non-universality.
Let us remark here that the extra assumption $n\mod (d+1)=0$
which we used in our construction, is not essential and
can be easily removed by letting $G_n$ to be a disjoint union
of $\lfloor n/(d+1)\rfloor-1$ cliques and a $d$--regular graph of size
$n-(d+1)\lfloor n/(d+1)\rfloor+(d+1)$
having arbitrary topology.

\bigskip

To prove the structural non-universality, we now set $\xi$ to be a $\pm 1$ Rademacher random variable. 
We assume $d \ge 3$, as the case $d=2$ can be verified by inspection.
Again, to simplify the exposition, we will
suppose that $n$ is a multiple of $d+1$.
Consider the sequence of graphs $G_n$ introduced above, and let $G_n'$ be a sequence of randomly uniformly chosen $d$-regular graphs on $n$ vertices. Then it was proved in \cite{mckay1981expected} that $G_n'$ converges (in the Benjamini and Schramm sense \cite{benjamini2011recurrence}) to the infinite rooted $d$-regular tree $(T_d,o)$. If $X'=\frac{1}{\sqrt{d}}A_{G'_n}\circ W^\xi_n$ where $G_n'$ and $W^\xi_n$ are independent, the above argument and the local convergence of $G'_n$ imply that
\begin{equation}\label{eq: ESD tree}
    \frac{1}{n}\esp \tr((X_n')^{2k})\to \frac{1}{d^k}\sum_{u \in \mathcal{P}(o,T_d,2k)}\esp w_{u_1u_2}^{\xi}\cdots w_{u_{2p}u_1}^{\xi}=m((T_d,o),2k,\xi).
\end{equation}
Let $k=3$, then, by counting the shapes, we can compute
\begin{align*}
    |\mathcal{P}(o,T_d,6)|=3d(d-1)^2+6d(d-1)+2d(d-1)(d-2)+d,
\end{align*}
so
\begin{align*}
    \frac{d^3}{n}\esp \tr((X_n')^{6}) \to 3d(d-1)^2+6d(d-1)+2d(d-1)(d-2)+d.
\end{align*}
On the other hand, for the clique $C_d$, we have
\begin{align*}
    \frac{d^3}{n}\esp \tr(X_n^{6})=2d(d-1)^2+8d(d-1)+3d(d-1)(d-2)+d.
\end{align*}
Indeed, the shape $s=o\to v_1 \to v_2 \to v_3 \to v_2 \to v_1 \to o$ that originally appeared for $T_d$ and its contribution $d(d-1)^2$ has to be decomposed on whether $v_3=o$ for the clique. In this case, the shape is two laps on a triangle, and inverting the orientation of the second lap implies the contribution of $2d(d-1)$. If $v_3 \ne o$, then its contribution is $d(d-1)(d-2)$ and the result follows. 
Subtracting one from the other, we get
\begin{align*}
    m((C_d,o),6,\xi)-m((T_d,o),6,\xi)=\frac{d(d-1)}{d^3}.
\end{align*}
In particular, structural non-universality follows as $d \ge 2$.
\end{proof}

\bigskip
\bigskip

\end{document}